\documentclass[a4paper]{article}
\usepackage{amsmath}
\usepackage{amsthm}
\usepackage{amssymb}
\usepackage{titlefoot} 
\usepackage{booktabs}
\usepackage{subcaption}
\usepackage{tikz}
\usetikzlibrary{arrows}
\usepackage{hyperref}

\theoremstyle{plain}
\newtheorem{theorem}{Theorem}
\newtheorem{lemma}[theorem]{Lemma}

\newtheorem{proposition}[theorem]{Proposition}

\theoremstyle{definition}

\newtheorem{conjecture}[theorem]{Conjecture}

\theoremstyle{remark}
\newtheorem{remark}[theorem]{Remark}

\newcommand{\order}[1]{\ensuremath{\left\lvert#1\right\rvert}}
\newcommand{\mob}{M\"{o}bius }
\newcommand{\mobfn}[2]{\mu[#1,#2]}
\newcommand{\mobp}[1]{\mu[#1]} 
\newcommand{\emptyperm}{\epsilon}
\DeclareMathOperator{\cl}{Cl}
\DeclareMathOperator{\cv}{Cover}
\DeclareMathOperator{\gr}{Gr}
\newcommand{\fullcover}[1]{\cl(#1)}
\newcommand{\xfullcover}[2]{\cl_{#1}(#2)}
\newcommand{\cover}[1]{\cv(#1)}
\newcommand{\ground}[1]{\gr(#1)}
\newcommand{\xground}[2]{\gr_{#1}(#2)}
\newcommand{\e}{\mathrm{e}}
\newcommand{\inflateall}[2]{#1[#2]}
\newcommand{\inflatesome}[3]{#1_{#2}[#3]}
\newcommand{\permsetall}[2]{#1\langle#2\rangle}
\newcommand{\permsetsome}[3]{#1_{#2}\langle#3\rangle}
\newcommand{\permsetfixed}[3]{#1_{#2}\langle\langle{#3}\rangle\rangle}
\newcommand{\stronglyzeroset}{\mathcal{SZ}}
\newcommand{\dnode}[3]{\node (p#1) at (#2,#3)  {\scriptsize{#1}}}
\newcommand{\cnode}[5]{%
    \node (p#1) [#5, align=center] at (#2,#3)%
        {\scriptsize{#1 #4}%
        };%
}
\newcommand{\dline}[2]{
    \foreach \i in {#2} {
        \draw [->] (p#1)  -- (p\i);
    };
}
\newcommand{\zpmfr}{0.3995}
\makeatletter
\def\anonfootnote{\xdef\@thefnmark{}\@footnotetext}
\makeatother
\title{Intervals of permutations and the principal \mob function}
\author{
    Robert Brignall, 
    David Marchant
    \\
    \\
    \textit{School of Mathematics and Statistics}\\[-3pt]
    \textit{The Open University, Milton Keynes, MK7 6AA, UK}\\[10pt]
}
\amssubj{05A05} 

\begin{document}
\maketitle
\anonfootnote{\textit{Email:} \{robert.brignall; david.marchant\}@open.ac.uk}

\begin{abstract}  
    We show that the proportion of permutations
    of length $n$
    with principal \mob function equal to zero, $Z(n)$,
    is asymptotically bounded below by \zpmfr.
    
    If a permutation $\pi$ contains two 
    intervals of length 2,
    where one interval is an ascent and the other a descent,
    then we show that the value of the 
    principal \mob function $\mobfn{1}{\pi}$
    is zero, and we use this 
    result to find the lower bound
    for $Z(n)$.  
         
    We also show that if a permutation $\phi$
    has certain properties, then
    any permutation $\pi$ which contains
    an interval order-isomorphic to $\phi$
    has $\mobfn{1}{\pi} = 0$.    
    
    An expanded version of this paper, with two additional authors,
    is available at \url{https://arxiv.org/abs/1810.05449}.
\end{abstract}

\section{Introduction}

Let $\sigma$ and $\pi$ be permutations of natural numbers.
We say that $\pi$ \emph{contains} $\sigma$ if
there is a sub-sequence of points of $\pi$
that is order-isomorphic to $\sigma$.
As an example, $3624715$ contains $3142$
as the sub-sequences $6275$ and $6475$. 
If $\sigma$ is contained in $\pi$,
then we write $\sigma \leq \pi$.

The set of all permutations is a poset
under the partial order given by containment.
A closed interval $[\sigma, \pi]$ in a poset 
is the set defined by 
$\{\tau : \sigma \leq \tau \leq \pi\}$,
and a half-open interval $[\sigma, \pi)$
is the set $\{\tau : \sigma \leq \tau < \pi\}$.
The \mob function is defined recursively 
on an interval of a poset $[\sigma, \pi]$ as:
\[
\mobfn{\sigma}{\pi} 
=
\left\lbrace
\begin{array}{lll}
0 & \quad & \text{if $\sigma \not\leq \pi$}, \\
1 & \quad & \text{if $\sigma = \pi$}, \\
- \sum\limits_{\tau \in [\sigma, \pi)} \mobfn{\sigma}{\tau} & \quad & \text{otherwise.}
\end{array}
\right.
\]
From the definition of the \mob function,
it follows that 
if $\sigma < \pi$, then
$\sum_{\tau \in [\sigma, \pi]} \mobp{\sigma, \tau} = 0$.

In this paper we are mainly concerned with the 
\emph{principal \mob function} 
of a permutation $\pi$, written $\mobp{\pi}$,
where $\mobp{\pi} = \mobfn{1}{\pi}$.
We consider cases where the value
of the principal \mob function $\mobp{\pi}$
can be determined by examining small localities
of $\pi$.  In this section,
we use ``locally zero'' to describe
results of this form.
In Section~\ref{section-permutations-containing-a-specific-interval}
we define ``strongly zero'' permutations
which share some characteristics
with locally zero permutations.
We show that, asymptotically,
the proportion of permutations
where the principal \mob function
is zero
is bounded below by \zpmfr.
We also show that our results for the 
principal \mob function can be extended 
to intervals where the lower bound is not $1$.

The question of the \mob function in 
the permutation poset
was first raised by
Wilf~\cite{Wilf2002}.  
The first result was by Sagan and Vatter~\cite{Sagan2006}, 
who determined the \mob function 
on intervals of layered permutations.
Steingr\'{\i}msson and Tenner~\cite{Steingrimsson2010} found  
pairs of permutations $(\sigma, \pi)$ 
where $\mobfn{\sigma}{\pi} = 0$.

Burstein, Jel{\'{i}}nek, Jel{\'{i}}nkov{\'{a}} 
and Steingr{\'{i}}msson~\cite{Burstein2011} found
a recursion for the \mob function
for sum and skew decomposable permutations.
They used this to determine
the \mob function for separable permutations.
Their results 
for sum and skew decomposable permutations
implicitly include the first locally zero result, 
which is that, up to symmetry, 
if a permutation $\pi$ with length greater than two begins $12$,
then $\mobp{\pi} = 0$.

Smith~\cite{Smith2013},
found an explicit formula for the \mob function on the interval
$[1, \pi]$ for all permutations $\pi$ with a single descent.
Smith's paper includes a lemma, 
reproduced here as Lemma~\ref{Smith-lemma-triple-adjacencies},
which is that if a permutation $\pi$
contains an interval order-isomorphic to
$123$, then $\mobp{\pi}=0$.
This is the second instance in the literature
of a locally zero result.
The result in~\cite{Burstein2011} requires  
that the permutation starts with a particular sequence.
Smith's result is, in some sense, more general,
as the critical interval (123)
can occur in any position.
We will see later that this lemma
is the first instance of a
strongly zero result.

Smith~\cite{Smith2016}, 
has explicit expressions for the 
\mob function $\mobfn{\sigma}{\pi}$
when $\sigma$ and $\pi$ have the same number of descents.
In~\cite{Smith2016a}, Smith found
an expression that determines the
\mob function for all intervals in
the poset, although the expression
involves a rather complicated
double sum, which includes
$\sum_{\tau \in [\sigma, \pi)} \mobfn{\sigma}{\tau}$.

Brignall and Marchant~\cite{Brignall2017a}
showed that if the lower bound of an interval is indecomposable,
then the \mob function depends only on the indecomposable permutations 
contained in the upper bound, 
and used this result to find a fast polynomial algorithm
for finding $\mobp{\pi}$ where $\pi$
is an increasing oscillation.

\section{Definitions and notation}

An \emph{interval} of a permutation $\pi$ 
is a non-empty set of contiguous indices $i, i+1, \ldots, j$
where the set of values $\{\pi_i, \pi_{i+1}, \ldots, \pi_j\}$
is also contiguous.  
A permutation of length $n$ that only has intervals
of length $1$ or $n$ is a \emph{simple} permutation.
As an example, $419725386$ is a simple permutation.
This is shown in Figure~\ref{figure-example-simple-oppadj}.

An \emph{adjacency} in a permutation is an interval of length two.
If a permutation contains a monotonic interval 
of length three or more, then each sub-interval 
of length two is an adjacency.
As examples, $367249815$ has two adjacencies, $67$ and $98$;
and $1432$ also has two adjacencies, $43$ and $32$.
If an adjacency is ascending, then it is an 
\emph{up-adjacency}, otherwise it is a 
\emph{down-adjacency}.

If a permutation $\pi$ contains at least one
up-adjacency, and at least one down-adjacency,
then we say that $\pi$ has \emph{opposing adjacencies}.
An example of a permutation with
opposing adjacencies is $367249815$,
which is shown in Figure~\ref{figure-example-simple-oppadj}.
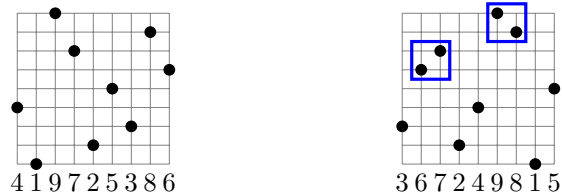
\begin{figure}[!h]
    \begin{center}
        \begin{subfigure}[t]{0.35\textwidth}
            \begin{center}
                \begin{tikzpicture}[scale=0.25]
                \draw[step=1cm,gray,very thin] (1,1) grid (9,9);
                \foreach \y [count=\x] in {4,1,9,7,2,5,3,8,6}
                {
                    \node [circle, draw, fill=black,inner sep=0pt, minimum width=4pt]
                    (\y) at (\x,\y) {};
                    \node[fill=none,draw=none] at (\x,0) {\y};
                }
                \end{tikzpicture}
            \end{center}
        \end{subfigure}
        \qquad      
        \begin{subfigure}[t]{0.35\textwidth}
            \begin{center}
                \begin{tikzpicture}[scale=0.25]
                \draw[step=1cm,gray,very thin] (1,1) grid (9,9);
                \foreach \y [count=\x] in {3,6,7,2,4,9,8,1,5}
                {
                    \node [circle, draw, fill=black,inner sep=0pt, minimum width=4pt]
                    (\y) at (\x,\y) {};
                    \node[fill=none,draw=none] at (\x,0) {\y};
                }
                \draw [color=blue, very thick] (1.5, 5.5) rectangle (3.5, 7.5);
                \draw [color=blue, very thick] (5.5, 7.5) rectangle (7.5, 9.5);    
                \end{tikzpicture}
            \end{center}
        \end{subfigure}
    \end{center}%
    \caption{A simple permutation and a permutation with opposing adjacencies.}
    \label{figure-example-simple-oppadj}
\end{figure}

A \emph{triple adjacency} in a permutation
is an interval of length three that is 
monotonic; that is, the interval is
order-isomorphic to $123$ or $321$.

A permutation that does not 
contain any adjacencies is
\emph{adjacency-free}.
Some early papers use the term ``strongly irreducible''
for what we call adjacency-free permutations.  
See, for example, Atkinson and Stitt~\cite{Atkinson2002}.

Given a permutation $\sigma$ with length $n$, and
permutations $\alpha_1, \ldots, \alpha_n$,
where at least one of the $\alpha_i$-s
is not the empty permutation, $\emptyperm$,
the \emph{inflation} of $\sigma$ by $\alpha_1, \ldots, \alpha_n$
is
the permutation found by 
removing the point $\sigma_i$
if $\alpha_i = \emptyperm$, and replacing $\sigma_i$
with an interval isomorphic to $\alpha_i$ otherwise.
Note that this is slightly different to the standard
definition of inflation, 
which does not allow inflation by the empty permutation.
We write inflations 
as
$\inflateall{\sigma}{\alpha_1, \ldots, \alpha_n}$.
As examples,
$\inflateall{3624715}{1,12,1,1,21,1,1}=367249815$,
and
$\inflateall{3624715}{\emptyperm,1,1,\emptyperm,1,\emptyperm,1}=3142$.
A \emph{proper inflation} is an inflation
$\inflateall{\sigma}{\alpha_1, \ldots, \alpha_n}$
where none of the  $\alpha_i$-s are the empty permutation.

In many cases we will be interested 
in permutations where most positions
are inflated by the singleton permutation $1$.
If $\sigma = 3624715$,
then 
we will write
$\inflateall{\sigma}{1,12,1,1,21,1,1} = 367249815$
as 
$\inflatesome{\sigma}{2,5}{12,21}$.
Formally, 
$\inflatesome{\sigma}{i_1, \ldots, i_k}{\alpha_1, \ldots, \alpha_k}$
is the inflation of $\sigma$, 
where $\sigma_{i_j}$ is inflated by $\alpha_{j}$ for
$j = 1, \ldots , k$; and all other positions of $\sigma$
are inflated by $1$.
We note here that inflations are not necessarily unique.
This is in contrast to the standard definition,
originally given in 
Albert and Atkinson~\cite{Albert2005},
where, essentially,
we have that every permutation
can be written as 
a unique inflation of a simple permutation.

If $\alpha$ is a permutation, then
the \emph{closure} of $\alpha$
is the set of permutations contained in $\alpha$,
including $\alpha$ itself and $\emptyperm$,
which we write as $\fullcover{\alpha}$.
If we have a permutation $\sigma$ of length $n$,
and permutations 
$\alpha_1, \ldots, \alpha_n$, then the
\emph{inflation set} 
$\permsetall{\sigma}{\alpha_1, \ldots, \alpha_n}$
is the set of all possible permutations that are inflations of 
$\sigma$, where each $\sigma_i$ is inflated
by an element of $\fullcover{\alpha_i}$.
In line with our definition of inflation,
we assume throughout this paper
that at least one of the permutations
chosen from 
$\fullcover{\alpha_1}, \ldots, \fullcover{\alpha_n}$
is non-empty, thus the set 
$\permsetall{\sigma}{\alpha_1, \ldots, \alpha_n}$
does not include the empty permutation.

We will mainly be 
discussing inflation sets where
most positions are inflated
by $\fullcover{1} = \{1, \emptyperm\}$.
We use the same style of notation that we use 
for inflations,
indicating the positions that are not inflated by 
an element of $\fullcover{1}$ as subscripts,
so, for example, we may write
$\permsetsome{\sigma}{\ell,r}{312,132}$
for the inflation set of $\sigma$,
where the $\ell$-th position 
is inflated by an element of $\fullcover{312}$,
the $r$-th position is inflated by 
an element of $\fullcover{132}$,
and all other positions are 
inflated by an element of $\fullcover{1}$.

Our argument includes discussing sets of permutations
that are an inflation of some $\sigma$,
where one position is inflated
by a specific permutation, and all other
positions are inflated by an element of $\fullcover{1}$.
The permutations in these sets are used as witnesses
to the presence of the specific permutation.
If we are inflating with the
specific permutation $\alpha$, 
then we
write the set of permutations as
$\permsetfixed{\sigma}{\ell}{\alpha}$,
where the permutation $\alpha$ inflates the $\ell$-th position,
and all other positions are inflated by 
an element of $\fullcover{1}$.

\section{Permutations with opposing adjacencies}

In this section our main result is to show that
if a permutation has opposing adjacencies, then
the value of the principal \mob function is zero.
We then show that if $\sigma$ is adjacency-free,
and $\pi$ contains an interval
order-isomorphic to a symmetry of $1243$, then
$\mobfn{\sigma}{\pi} = 0$.

We use an inductive proof.  The first,
rather trivial, step is to show that the base case holds.
We then consider some $\pi$ that has opposing adjacencies.
and divide the poset $[1, \pi)$
into four sets 
$L$,
$R$,
$L \cap R$ and 
$T = [1, \pi) \setminus (L \cup R )$,
and show that we can obtain
$\mobp{\pi}$ 
by summing over each set and then using inclusion-exclusion.
We show an example of the sets for $\pi = 346215$ 
in Figure~\ref{figure-PMF-sets}.
\begin{figure}[h!]
    \begin{center}
        \begin{tikzpicture}[xscale=1.7,yscale=1.3]
        \dnode{346215}{3.0}{5};
        \dnode{23514}{1.0}{4};        
        \dnode{34521}{2.5}{4};        
        \dnode{34215}{3.5}{4};        
        \dnode{35214}{5.0}{4};        
        \dnode{1243}{0.0}{3};
        \dnode{2341}{1.0}{3};
        \dnode{2314}{2.0}{3};
        \dnode{2413}{3.0}{3};
        \dnode{3421}{4.0}{3};
        \dnode{3214}{5.0}{3};
        \dnode{4213}{6.0}{3};
        \dnode{123}{0.5}{2};
        \dnode{132}{1.5}{2};
        \dnode{231}{2.5}{2};
        \dnode{213}{3.5}{2};
        \dnode{312}{4.5}{2};
        \dnode{321}{5.5}{2};
        \dnode{12}{2.5}{1};
        \dnode{21}{3.5}{1};
        \dnode{1}{3.0}{0};
        \dline{346215}{23514,34521,34215,35214};
        \dline{23514}{1243,2341,2314,2413};
        \dline{34521}{2341,3421};
        \dline{34215}{2314,3421,3214};
        \dline{35214}{2413,3421,3214,4213};
        \dline{1243}{123,132};
        \dline{2341}{123,231};
        \dline{2314}{123,231,213};
        \dline{2413}{132,231,213,312};
        \dline{3421}{231,321};
        \dline{3214}{213,321};
        \dline{4213}{213,312,321};
        \dline{123}{12};
        \dline{132}{12,21};
        \dline{213}{12,21};
        \dline{231}{12,21};
        \dline{312}{12,21};
        \dline{321}{21};
        \dline{12}{1};
        \dline{21}{1};
        \draw [thick, draw=red, fill=red, fill opacity=0.05] 
        plot [smooth cycle] coordinates {
            ( 1.0,  4.2) 
            (-0.3,  3.0)
            ( 0.2,  2.0)
            ( 2.9, -0.2)
            ( 4.7,  2.1)
        };
        \draw [thick, draw=blue, fill=blue, fill opacity=0.05] 
        plot [smooth cycle] coordinates {
            ( 5.0,  4.2) 
            ( 6.3,  3.0)
            ( 5.8,  2.0)
            ( 3.1, -0.2)
            ( 1.3,  2.1)
        };
        \draw [thick, draw=cyan, fill=cyan, fill opacity=0.05] 
        plot [smooth cycle] coordinates {
            (3,4.3) 
            (2,4) 
            (3,3.7) 
            (4,4)
        };
        \node (L) [red] at (0.5 ,0.5) {$L$};
        \draw [-triangle 45, red, dashed] (L) -- (1.1,1.1);
        \node (R) [blue] at (5.5, 0.5) {$R$};
        \draw [-triangle 45, blue, dashed] (R) -- (4.9,1.1);
        \node (T) [cyan] at (1.5,4.75) {$T$};
        \draw [-triangle 45, cyan, dashed] (T) -- (2.3,4.15);
        \end{tikzpicture}
    \end{center}
    \caption{Four subsets of $[1, 346215)$:
        $L$ in red, $R$ in blue, $L \cap R$, and $T$ in cyan.}
    \label{figure-PMF-sets}
\end{figure}
\subsection{The principal \mob function of permutations with an opposing adjacency}

Our main theorem in this section is:
\begin{theorem}
    \label{theorem-PMF-opposing-adjacencies}
    If $\pi$ has opposing adjacencies, then
    $\mobp{\pi} = 0$.
\end{theorem}
\begin{proof}
    Note first that if $\pi$ has
    opposing adjacencies, then it must have
    length at least four.  It is simple
    to see that all permutations of
    length four with opposing adjacencies
    are symmetries of $1243$, and that 
    $\mobp{1243} = 0$.
    
    Assume now that 
    Theorem~\ref{theorem-PMF-opposing-adjacencies}
    applies to all permutations with length less than some $n > 4$.    
    Let $\pi$ be a permutation of length $n$ 
    with opposing adjacencies.
    Choose an up-adjacency and a down-adjacency.  
    Without loss of generality, 
    we can assume, by symmetry, that the first adjacency chosen 
    is an up-adjacency.
    
    Let $\gamma$ be the permutation formed
    by replacing the two chosen adjacencies in $\pi$
    by $1$, and retaining all other points of $\pi$.
    Then we can write 
    $
    \pi = \inflatesome{\gamma}{\ell,r}{12,21}
    $,
    where $\ell$ is the index of the first point
    of the first adjacency chosen in $\pi$,
    and 
    $r$ is one less than the index 
    of the first point of the second adjacency chosen in $\pi$.  
    As an example, with $\pi=3\mathbf{67}24\mathbf{98}15$
    we would have $\ell=2$, $r=5$,
    and 
    $\pi = \inflatesome{3624715}{\ell,r}{12,21}$.
    It is easy to see that 
    with $\ell$ and $r$ fixed, $\gamma$ is unique.
    
    Let $\lambda = \inflatesome{\gamma}{\ell}{12}$
    and let $\rho = \inflatesome{\gamma}{r}{21}$.
    
    We define four (overlapping) subsets of 
    $[1, \pi)$ as follows:
    \begin{align*}
    L & = [1, \lambda]  \\
    R & = [1, \rho]  \\
    G & = L \cap R   \\
    T & = [1, \pi) \setminus ( L \cup R )    
    \end{align*}
    Since $\lambda$ and $\rho$ are both contained in $\pi$, 
    it is easy to see that
    \begin{align}
    \label{equation-PMS-two-adj-four-sums}
    \mobp{\pi} = 
    - \sum_{\tau \in L} \mobp{\tau}
    - \sum_{\tau \in R} \mobp{\tau}
    - \sum_{\tau \in T} \mobp{\tau}
    + \sum_{\tau \in G} \mobp{\tau}
    \end{align}
    To prove Theorem~\ref{theorem-PMF-opposing-adjacencies}
    it is sufficient to show that each of the
    four sums in Equation~\ref{equation-PMS-two-adj-four-sums}
    is zero.
    
    Consider first $\sum_{\tau \in L} \mobp{\tau}$.
    This is plainly zero from the definition of the 
    \mob function, and the same argument
    applies to $\sum_{\tau \in R} \mobp{\tau}$.
    
    Now consider $\sum_{\tau \in T} \mobp{\tau}$.
    We claim that any permutation $\tau \in T$ must have
    opposing adjacencies, and so, using our
    inductive hypothesis, $\mobp{\tau} = 0$.
    To justify our claim, 
    let $\tau$ be an element of $[1, \pi)$.
    If $\tau$ does not contain an up-adjacency,
    then $\tau$ must be in 
    $\permsetsome{\gamma}{\ell,r}{1,21}$,
    and so is in $R$.
    Similarly, if $\tau$ does not contain
    a down-adjacency, then $\tau \in L$.
    Thus if $\tau \in T$, then
    $\tau$ has an opposing adjacency.
    
    Finally, we consider $\sum_{\tau \in G} \mobp{\tau}$.
    We partition $G$ into two disjoint sets:
    \begin{align*}
    G_\gamma & = [1, \gamma]  \\
    G_x & =  L \cap R \setminus G_\gamma 
    \end{align*}
    From the construction of $L$ and $R$ 
    it is easy to see that
    $G_\gamma \subseteq L \cap R$, 
    so $G_\gamma$ and $G_x$ are well-defined, and we have
    $\sum_{\tau \in G} \mobp{\tau} 
    =
    \sum_{\tau \in G_\gamma} \mobp{\tau} 
    +
    \sum_{\tau \in G_x} \mobp{\tau} 
    $.
    In some cases we observe that
    $G_\gamma = L \cap R$, 
    but this is not true in general as, for example, 
    when $\pi = 53128746$, we have
    $G_x = \{ 4312 \}$.
    
    From the definition of the \mob function, 
    $\sum_{\tau \in G_\gamma} \mobp{\tau} = 0$,
    so to complete the proof of 
    Theorem~\ref{theorem-PMF-opposing-adjacencies}
    we simply need to show that
    $\sum_{\tau \in G_x} \mobp{\tau} = 0$.
    
    We claim that every permutation $\tau$ in 
    $G_x$ has an opposing adjacency, and so
    has $\mobp{\tau} = 0$.  
    
    Since 
    $L =        \permsetsome{\gamma}{\ell}{12}$, and
    $G_\gamma = \permsetsome{\gamma}{\ell}{1}$,
    it follows that, as 
    $\tau \in L \setminus G_\gamma$,
    then $\tau \in \permsetfixed{\gamma}{\ell}{12}$,
    so $\tau$ contains an up-adjacency.
    Similarly,  if
    $\tau \in R \setminus G_\gamma$,
    then $\tau \in \permsetfixed{\gamma}{r}{21}$,
    so $\tau$ contains a down-adjacency.
    Thus if $\tau \in G_x$, then $\tau$ has 
    opposing adjacencies, and
    so, by the inductive hypothesis,
    $\mobp{\tau} = 0$,
    and thus 
    $\sum_{\tau \in G_x} \mobp{\tau} = 0$.
\end{proof}
\subsection{Extending Theorem~\ref{theorem-PMF-opposing-adjacencies}}

It is natural to ask if we can extend 
Theorem~\ref{theorem-PMF-opposing-adjacencies}
to handle cases where 
the lower bound of the 
interval is not $1$.
This is not possible in general, as
if we take any permutation $\sigma \neq 1$,
and inflate any two distinct points in positions
$\ell$ and $r$ by $12$ and $21$
respectively, then
$\pi = \inflatesome{\sigma}{\ell,r}{12,21}$
has opposing adjacencies,
but 
$\mobfn{\sigma}{\pi} = 1$, as
can
be deduced from 
Figure~\ref{figure-extending-PFM-oppadj-fails}.
    \begin{figure}[!h]
    \begin{center}
        \begin{tikzpicture}[xscale=1,yscale=1]
        \node (n12-21) at ( 0, 3) {$\inflatesome{\sigma}{\ell,r}{12,21}$};
        \node (n12-1)  at (-2, 2) {$\inflatesome{\sigma}{\ell,r}{12,1}$};  
        \node (n1-21)  at ( 2, 2) {$\inflatesome{\sigma}{\ell,r}{1,21}$};  
        \node (n1-1)   at ( 0, 1) {$\inflatesome{\sigma}{\ell,r}{1,1}$};  
        \draw (n12-21) -- (n1-21);
        \draw (n12-21) -- (n12-1);
        \draw (n12-1) --  (n1-1);
        \draw (n1-21) --  (n1-1);
        \end{tikzpicture}
        \qquad
        \begin{tikzpicture}[xscale=1,yscale=1]
        \node (n12-21) at ( 0, 3) {$1$};
        \node (n12-1)  at (-2, 2) {$-1$};  
        \node (n1-21)  at ( 2, 2) {$-1$};  
        \node (n1-1)   at ( 0, 1) {$1$};  
        \draw (n12-21) -- (n1-21);
        \draw (n12-21) -- (n12-1);
        \draw (n12-1) --  (n1-1);
        \draw (n1-21) --  (n1-1);
        \end{tikzpicture}
    \end{center}
    \caption{The Hasse diagram of the interval 
        $[\sigma, \inflatesome{\sigma}{\ell,r}{12,21}]$,
        and the corresponding values of the \mob function.}
    \label{figure-extending-PFM-oppadj-fails}
\end{figure}
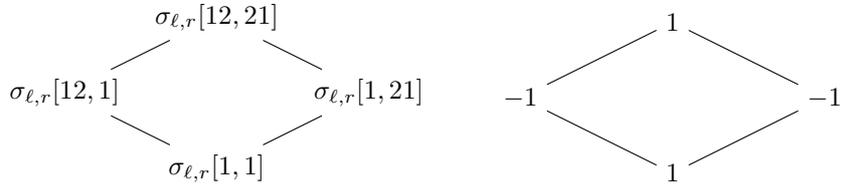

Although we do not have a general extension of
Theorem~\ref{theorem-PMF-opposing-adjacencies},
we can show that:
\begin{theorem}
    \label{theorem-MF-opposing-adjacencies}
    If $\sigma$ is adjacency-free,
    and $\pi$ contains an interval 
    order-isomorphic to a symmetry of $1243$,
    then $\mobfn{\sigma}{\pi} = 0$.
\end{theorem}
\begin{proof}
    First note that if $\sigma \not\leq \pi$,
    then $\mobfn{\sigma}{\pi} = 0$ 
    from the definition of the \mob function.
    Further, since $\sigma$ is adjacency-free,
    we cannot have $\sigma = \pi$.
    
    We can now assume that $\sigma < \pi$.
    Without loss of generality 
    we can also assume, by symmetry, that the
    interval in $\pi$ is
    order-isomorphic to $1243$.
    
    We start by claiming that, for any permutation
    $\sigma$ which is adjacency-free, 
    and any $c$ with $1 \leq c \leq \order{\sigma}$,
    we have 
    $\mobfn{\sigma}{\inflatesome{\sigma}{c}{1243}} = 0$.
    The Hasse diagram of the interval 
    $[\sigma, \inflatesome{\sigma}{c}{1243}]$ is shown in
    Figure~\ref{figure-hasse-interval-inflate-1243}.
    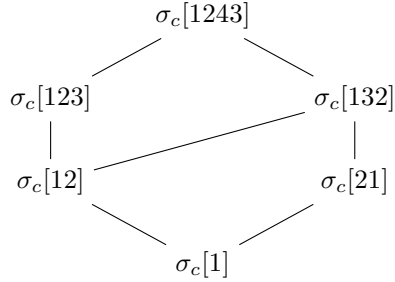
\begin{figure}[!h]
        \begin{center}
            \begin{tikzpicture}[xscale=1,yscale=1.1]
            \node (n1243) at ( 0, 3) {$\inflatesome{\sigma}{c}{1243}$};
            \node (n123)  at (-2, 2) {$\inflatesome{\sigma}{c}{123}$};  
            \node (n132)  at ( 2, 2) {$\inflatesome{\sigma}{c}{132}$};  
            \node (n12)   at (-2, 1) {$\inflatesome{\sigma}{c}{12}$};  
            \node (n21)   at ( 2, 1) {$\inflatesome{\sigma}{c}{21}$};  
            \node (n1)    at ( 0, 0) {$\inflatesome{\sigma}{c}{1}$};  
            \draw (n1243) -- (n123);
            \draw (n1243) -- (n132);
            \draw (n123) --  (n12);
            \draw (n132) --  (n12);
            \draw (n132) --  (n21);
            \draw (n12)  --  (n1);
            \draw (n21)  --  (n1);            
            \end{tikzpicture}
        \end{center}
        \caption{The Hasse diagram of the interval 
            $[\sigma, \inflatesome{\sigma}{c}{1243}]$.}
        \label{figure-hasse-interval-inflate-1243}
    \end{figure}
    From the definition of the \mob function, we have
    $\mobfn{\sigma}{\inflatesome{\sigma}{c}{1}} = 1$,
    $\mobfn{\sigma}{\inflatesome{\sigma}{c}{12}} = -1$, 
    $\mobfn{\sigma}{\inflatesome{\sigma}{c}{21}} = -1$, 
    $\mobfn{\sigma}{\inflatesome{\sigma}{c}{123}} = 0$,
    and
    $\mobfn{\sigma}{\inflatesome{\sigma}{c}{132}} = 1$,
    and so
    $\mobfn{\sigma}{\inflatesome{\sigma}{c}{1243}} = 0$,
    and thus our claim is true.
    
    Our argument now follows a similar pattern to
    that used by Theorem~\ref{theorem-PMF-opposing-adjacencies},
    and we restrict ourselves to highlighting the differences.
    
    Assume that $\pi$ is a proper inflation of $\sigma$,
    with length greater than $\order{\sigma} + 4$, 
    and $\pi$ contains an interval order-isomorphic to $1243$.
    Let $\gamma$ be the permutation formed by replacing
    an occurrence of $1243$ in $\pi$ by $12$, so if 
    $\ell$ is the position of the first point 
    of the $1243$ selected, then
    $\inflatesome{\gamma}{\ell,\ell+1}{12,21} = \pi$.
    Let $\lambda = \inflatesome{\gamma}{\ell,\ell+1}{12,1}$, 
    and let $\rho = \inflatesome{\gamma}{\ell,\ell+1}{1,21}$;
    Define sets 
    $L = [\sigma, \lambda]$,
    $R = [\sigma, \rho]$,
    $G_\gamma = [\sigma, \gamma]$,
    $G_x = L \cap R \setminus G_\gamma$, and
    $T = [\sigma, \pi) \setminus (L \cup R)$.
    
    Similarly to Theorem~\ref{theorem-PMF-opposing-adjacencies}, 
    we have 
    \[
    \mobfn{\sigma}{\pi} = 
    - \sum_{\tau \in L} \mobfn{\sigma}{\tau}
    - \sum_{\tau \in R} \mobfn{\sigma}{\tau}
    - \sum_{\tau \in T} \mobfn{\sigma}{\tau}
    + \sum_{\tau \in G_\gamma} \mobfn{\sigma}{\tau}
    + \sum_{\tau \in G_x} \mobfn{\sigma}{\tau},
    \]
    and the sums over the sets 
    $L$, $R$ and $G_\gamma$ are obviously zero.
    Using similar arguments to 
    Theorem~\ref{theorem-PMF-opposing-adjacencies}, 
    we can see that every permutation $\tau$
    in $T$ or $G_x$ contains
    an interval order-isomorphic to $1243$,
    and so by the inductive hypothesis,
    has $\mobfn{\sigma}{\tau} = 0$,
    and thus we have 
    $\mobfn{\sigma}{\pi} = 0$.
\end{proof}    

Although we cannot finds a general 
extension to
Theorem~\ref{theorem-PMF-opposing-adjacencies},
we can find a necessary condition for 
a proper inflation of certain permutations
to have a \mob function value of zero.
This is
\begin{lemma}
    If $\sigma$ is adjacency-free,
    and 
    $\pi = \inflateall{\sigma}{\alpha_1, \ldots, \alpha_n}$ 
    is a proper inflation of $\sigma$,
    then $\mobfn{\sigma}{\pi} = 0$ implies that
    at least one $\alpha_i \not \in \{1, 12, 21\}$.    
\end{lemma}
\begin{proof}
    Assume that every $\alpha_i \in \{1, 12, 21\}$.
    Let $k$ be the number of $\alpha_i$-s that are not equal to 1,
    and 
    let $j_1, \ldots , j_k$
    be the indexes ($i$-s) where $\alpha_i \neq 1$,
    so 
    $\pi = 
    \inflatesome{\sigma}{j_1, \ldots, j_k}
    {\alpha_{j_1}, \ldots, \alpha_{j_k}}$.
    
    Then every permutation in the interval
    $[\sigma, \pi]$ has a unique representation
    as 
    $\inflatesome{\sigma}{j_1, \ldots, j_k}{\upsilon_1, \ldots , \upsilon_k}$,
    where
    \begin{align*}
    \upsilon_i & \in
    \begin{cases}
    \{ 1, 12 \} & \text{if } \alpha_{j_i} = 12, \\
    \{ 1, 21 \} & \text{if } \alpha_{j_i} = 21. \\
    \end{cases}
    \end{align*}
    So each position $j_i$ can be inflated 
    by one of two permutations, and thus
    there is an obvious isomorphism
    between permutations in the interval
    $[\sigma, \pi]$  
    and binary numbers with $k$ bits.
    It follows that the poset can be represented as
    a Boolean algebra, and so by a 
    well-known result 
    (see, for instance,
    Example 3.8.3 in Stanley~\cite{Stanley2012}),
    $\mobfn{\sigma}{\pi} = (-1)^{\order{\pi} - \order{\sigma}}$.
    Thus if $\mobfn{\sigma}{\pi} = 0$, at least one 
    $\alpha_i \not \in \{1, 12, 21\}$.
\end{proof}

\section{Permutations containing a specific interval}
\label{section-permutations-containing-a-specific-interval}

In this section we show that if a permutation $\phi$ 
meets certain requirements,
then any permutation $\pi$ with an 
interval order-isomorphic to $\phi$
has $\mobp{\pi} = 0$.  

Recall that if $\pi$ is a permutation, then 
a permutation $\sigma$ is \emph{covered} by $\pi$
if $\sigma < \pi$, and there is no 
permutation $\tau$ such that $\sigma < \tau < \pi$.
The set of permutations covered by $\pi$
is the \emph{cover} of $\pi$, written 
$\cover{\pi}$.

As with opposing adjacencies, our proof is inductive.
For the induction step with some permutation $\pi$,
our approach is to partition the poset
into subsets 
and then show that 
the sum of the principal \mob function values
over the permutations
in each subset is zero.
As an example, 
Figure~\ref{figure-PMF-strongly-zero-sets}
shows how we would partition
the interval $[1, 1324657)$.  
In this case we have four subsets:
$P$,
$L_1$,
$L_2$ and 
$R$.
\begin{figure}[h!]
    \begin{center}
        \begin{tikzpicture}[xscale=1.5,yscale=1.5]
        \dnode{1324657}{3}{6};
        
        \cnode{132456}{0}{5}{$L_2$}{cyan};
        \cnode{123546}{1.5}{5}{$L_1$}{blue};
        \cnode{213546}{6}{5}{$R$}{red};
        \cnode{132546}{4.5}{5}{$P$}{magenta};
        \cnode{132465}{7}{5}{$R$}{red};
        
        \dline{1324657}{132546,123546,132456,132465,213546}        
        
        \cnode{12345}{1}{4}{$L_1$}{blue};
        \cnode{21345}{0}{4}{$L_2$}{cyan};
        \cnode{12435}{4}{4}{$P$}{magenta};
        \cnode{13245}{3}{4}{$P$}{magenta};
        \cnode{12354}{2}{4}{$L_1$}{blue};
        \cnode{21435}{5}{4}{$P$}{magenta};
        \cnode{21354}{7}{4}{$R$}{red};
        \cnode{13254}{6}{4}{$P$}{magenta};
        
        \dline{132546}{12435,13245,21435,13254};
        \dline{123546}{12345,12435,12354};
        \dline{132456}{12345,21345,13245};
        \dline{132465}{13245,12354,21354,13254};
        \dline{213546}{21345,12435,21435,21354};
        
        \cnode{1234}{2}{3}{$P$}{magenta};
        \cnode{2134}{3}{3}{$P$}{magenta};
        \cnode{1324}{4}{3}{$P$}{magenta};
        \cnode{1243}{5}{3}{$P$}{magenta};
        \cnode{2143}{6}{3}{$P$}{magenta};
        
        \dline{12435}{1234,1324,1243};
        \dline{13245}{1234,2134,1324};
        \dline{21435}{2134,1324,2143};
        \dline{13254}{1324,1243,2143};
        \dline{12345}{1234};
        \dline{12354}{1234,1243};
        \dline{21345}{1234,2134};
        \dline{21354}{2134,1243,2143};  
        
        \cnode{123}{3}{2}{$P$}{magenta};      
        \cnode{213}{4}{2}{$P$}{magenta};      
        \cnode{132}{5}{2}{$P$}{magenta};      
        
        \dline{1234}{123};
        \dline{2134}{123,213};
        \dline{1324}{123,213,132};
        \dline{1243}{123,132};
        \dline{2143}{213,132};
        
        \cnode{12}{4}{1}{$P$}{magenta};
        \cnode{21}{5}{1}{$P$}{magenta};
        
        \dline{123}{12};
        \dline{213}{12,21};
        \dline{132}{12,21};
        
        \cnode{1}{4}{0}{$P$}{magenta};
        
        \dline{12}{1};
        \dline{21}{1};
        \end{tikzpicture}
    \end{center}
    \caption{Four subsets of $[1, 1324657)$:
        $P$, 
        $L_1$, 
        $L_2$, 
        and 
        $R$.}
    \label{figure-PMF-strongly-zero-sets}
\end{figure}
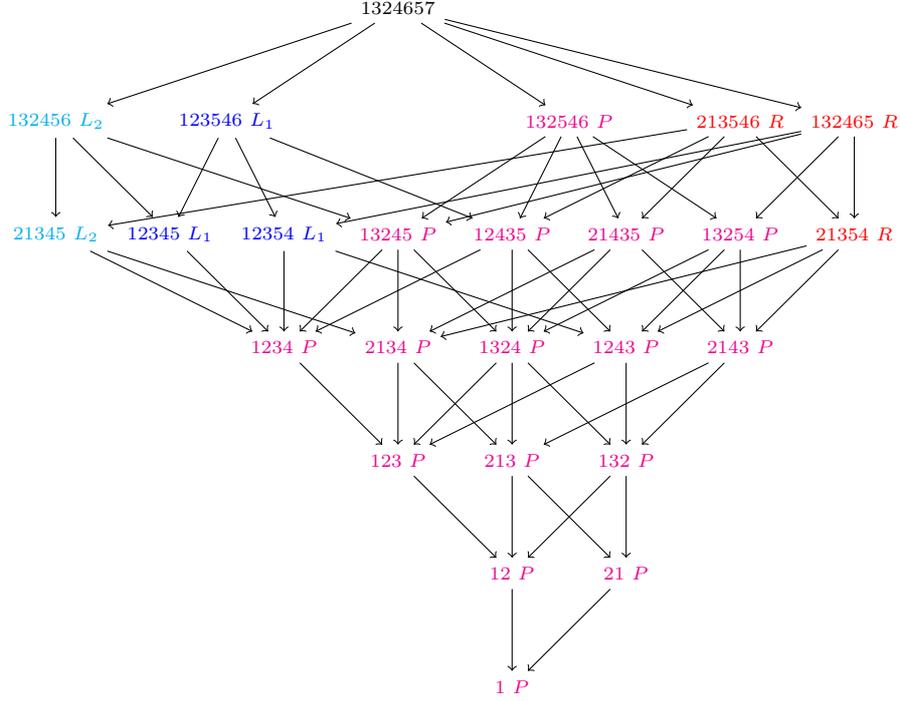
The partitioning is based around the permutations that
are covered by $\phi$.  
One of these permutations is the ``core'', 
and the associated subset is $P$.
Further subsets are specified by the 
other permutations in the cover of $\phi$,
making sure that no permutation
is included in more than one subset,
and we label these subsets $L_1, \ldots L_k$.
Once we have iterated through 
the subsets defined by the permutations 
in the cover of $\phi$, any remaining
permutations are placed in a final subset $R$.
Our proof then relies on showing that,
for each set $S$, 
$\sum_{\tau \in S} \mobp{\tau} = 0$.

We will need to use a 
lemma from Smith~\cite{Smith2013}:

\begin{lemma}[{%
        Smith~\cite[Lemma 1]{Smith2013}}]
    \label{Smith-lemma-triple-adjacencies}
    If a permutation $\pi$
    contains a triple adjacency then $\mobp{\pi} = 0$.
\end{lemma}

Given a permutation $\pi$, it is sometimes 
possible to determine the value of $\mobp{\pi}$
by considering small localities of $\pi$.  
As examples, if $\pi$ contains a monotonic interval
of length three or more,
then 
by Lemma~\ref{Smith-lemma-triple-adjacencies},
$\mobp{\pi} = 0$.  
Theorem~\ref{theorem-PMF-opposing-adjacencies}
in this paper is another example,
as the presence of opposing adjacencies 
guarantees that $\mobp{\pi} = 0$.
The two instances mentioned can be rephrased 
in terms of inflations, so if a permutation $\pi$
can be written as
$\inflatesome{\gamma}{c}{123}$,
$\inflatesome{\gamma}{c}{321}$,
$\inflatesome{\gamma}{\ell,r}{12,21}$,
or
$\inflatesome{\gamma}{\ell,r}{21,12}$
with 
$1 \leq \ell < r \leq \order{\gamma}$, 
and
$1 \leq c \leq \order{\gamma}$,
then we know, from 
Lemma~\ref{Smith-lemma-triple-adjacencies}
and
Theorem~\ref{theorem-PMF-opposing-adjacencies},
that $\mobp{\pi} = 0$.

Let $\stronglyzeroset$ be the set of
permutations such that
if any permutation $\pi$ contains
an interval order-isomorphic 
to some $\phi \in \stronglyzeroset$
then $\mobp{\pi} = 0$.  
We know that
$\stronglyzeroset$ is non-empty,
since $123$, $321$ 
and all permutations with opposing adjacencies
are elements of $\stronglyzeroset$.
If a permutation $\phi$
is an element of $\stronglyzeroset$, then we
say that $\phi$ is 
\emph{strongly zero}.

There are cases where we can determine that
the principal \mob function of a permutation is zero
by examining part of the
permutation, but the
permutation is not strongly zero.
As an example,
if a permutation $\pi$, with $\order{\pi} > 2$,
begins $12$, then
as a
consequence of  
Propositions 1 and 2 in 
Burstein, 
Jel{\'{i}}nek, 
Jel{\'{i}}nkov{\'{a}} and 
Steingr{\'{i}}msson~\cite{Burstein2011}
(first stated explicitly as Lemma 4 in 
Brignall and Marchant~\cite{Brignall2017a}),
$\mobp{\pi} = 0$.
Such a permutation is not strongly zero,
since $12 \not \in \stronglyzeroset$.

We need one further definition before we can define ``core''
and proceed to the statement of our theorem.

Let $L = \{ \lambda_1, \ldots, \lambda_n \}$
be a set of permutations.
The \emph{ground} of $L$, $\ground{L}$,
is the set of permutations formed
by taking the union of the closure of each permutation in 
$L$, and then removing any permutation
that is strongly zero, so
\[
\ground{L}
=
\left(
\bigcup_{\lambda \in L} \fullcover{\lambda}
\right)
\setminus \stronglyzeroset.
\]
As an example, if
$L = \{ 1243, 2134\}$,
then
we have
\begin{align*}
&&\fullcover{1243} & = \{1243, 132, 123, 12, 21, 1, \emptyperm \}, \\
&&\qquad\fullcover{2134} & = \{2134, 123, 213, 12, 21, 1, \emptyperm \},  \\
&\text{and} &\stronglyzeroset & = \{ 123, 321, 1243, 2134, \ldots \}, \\
&\text{so} &\qquad \ground{L}       & = \{132, 213, 12, 21, 1, \emptyperm \}.
\end{align*}

Let $\phi$ be a permutation, 
and let $\psi$ be an element of $\cover{\phi}$.
Let $L = \cover{\phi} \setminus \psi$.
We say that $\psi$ is a \emph{core} of $\phi$
if every permutation in the ground of $L$
is contained in $\psi$, i.e., 
$\ground{L} \subseteq \fullcover{\psi}$.

If a permutation $\phi$ has a core $\psi$,
then this means that
every permutation in
$\cover{\phi} \setminus \psi$
must be strongly zero.
As a consequence, any permutation $\phi$
where $\cover{\phi}$ contains more than one
permutation that is not strongly zero
does not have a core.
Further, if $\phi$ contains exactly one
permutation $\psi$ that is not strongly zero,
then either $\psi$ is the core, as it meets the
requirement that
$\ground{\cover{\phi}} \subseteq \fullcover{\psi}$,
or $\phi$ does not have a core.

A further consequence of this definition,
which we will use in our proof,
is that
if $\phi$ has a core $\psi$,
then the closure of $\phi$ 
is the union of 
$\phi$,
permutations in the closure of $\psi$,
and
permutations that are structurally zero,
so
$
\fullcover{\phi} \setminus ( \phi \cup \fullcover{\psi} ) 
\subset \stronglyzeroset
$.

It is possible for a permutation $\phi$ to have 
more than one core.
For our purposes, all we require is that a core
exists, and henceforth 
we will refer to ``the core''
of a permutation.

As an example,
let $\phi = 21354$.  Then 
$\cover{\phi}  = \{1243, 2143, 2134 \} $.
The only possibility for the core is $2143$, so we have
\begin{align*}
&&\fullcover{2143} &= \{2143, 132, 213, 12, 21, 1, \emptyperm \} \\
&\text{and}& \ground{\cover{21354} \setminus 2143} 
& =
\{132, 213, 12, 12, 1, \emptyperm  \},\\
&\text{so}&\ground{\cover{21354}} \setminus 2143
& \subseteq \fullcover{2143}
\end{align*}
and thus $2143$ is the core of $21354$.

We now have our final definition in this section.
Let $\phi$ be any permutation.
We say that $\phi$ is \emph{nice}  
if
$\mobp{\phi} = 0$ and 
$\phi$ has a core.
Continuing with our running example of
$\phi = 21354$, we know that
$\phi$ has a core of $2143$.  Since
$\mobp{21354}  = 0$,
this gives us that $21354$ is nice.

We are now in a position to state our main theorem for this section.
\begin{theorem}
    \label{theorem-PMF-strongly-zero-interval}
    If $\phi$ is a nice permutation,
    and $\pi$ is any permutation
    containing an interval order-isomorphic to $\phi$,
    then $\mobp{\pi} = 0$, thus
    $\phi \in \stronglyzeroset$.
\end{theorem}
\begin{proof}
    We proceed by induction.  
    First, if $\phi = \pi$, then by definition
    $\mobp{\pi} = 0$.
    Now assume that, for a given $\phi$,
    Theorem~\ref{theorem-PMF-strongly-zero-interval}
    is true for all permutations with length
    less than some $n$, where $n > \order{\phi}$.
    Let $\pi$ be a permutation of length $n$
    that contains at least one interval
    order-isomorphic to $\phi$. 
    Choose one of the intervals order-isomorphic to $\phi$,
    and let $\gamma$ be the permutation obtained
    by replacing the chosen interval with a single point.
    Let $c$ be the index of the first point of the chosen
    interval in $\pi$,
    so $\pi = \inflatesome{\gamma}{c}{\phi}$.
    
    Let $\psi$ be the core of $\phi$;
    and let $\lambda_1, \ldots, \lambda_k$
    be the permutations covered by 
    $\phi$ excluding $\psi$, 
    i.e., $\cover{\phi} \setminus \psi$.    
    
    Our approach is to divide the
    poset $[1, \pi)$ into disjoint sets
    $P, L_1, \ldots , L_k, R$, as
    defined below, 
    and then show that for each set 
    $S \in \{ P, L_1, \ldots , L_k, R \}$,
    $\sum_{\tau \in S} \mobp{\tau} = 0$, 
    which gives us $\mobp{\pi} = 0$.
    
    Define sets of permutations as follows:
    \begin{align*}
    & &
    P   & = \fullcover{\inflatesome{\gamma}{c}{\psi}} 
    \\
    L_1^\prime & = \fullcover{\inflatesome{\gamma}{c}{\lambda_1}} 
    &
    L_1 & = L_1^\prime \setminus ( P ) 
    \\
    L_2^\prime & = \fullcover{\inflatesome{\gamma}{c}{\lambda_2}} 
    &
    L_2 & = L_2^\prime \setminus ( P \cup L_1^\prime ) 
    \\
    \vdots & \mathrel{\phantom{=}} \vdots
    &
    \vdots & \mathrel{\phantom{=}} \vdots
    \\
    L_k^\prime & = \fullcover{\inflatesome{\gamma}{c}{\lambda_k}} 
    &
    L_k & = L_k^\prime \setminus ( P \cup L_1^\prime \cup \ldots \cup L_{k-1}^\prime) 
    \\ 
    & &
    R & = [1, \pi) \setminus (P \cup L_1^\prime \cup \ldots \cup L_{k}^\prime)
    \end{align*}
    It is easy to see that in
    $\{ P, L_1, \ldots L_k, R\}$
    the intersection of any distinct pair of sets is empty, and that
    $
    [1, \pi)
    =
    P \cup L_1 \cup \ldots \cup L_k \cup R
    $.
    Thus 
    \begin{align}
    \label{equation-PMF-nice-sums}
    \mobp{\pi}
    =
    - \sum_{\tau \in P} \mobp{\tau}
    - \sum_{\tau \in L_1} \mobp{\tau}
    -
    \ldots
    - \sum_{\tau \in L_k} \mobp{\tau}
    - \sum_{\tau \in R} \mobp{\tau},    
    \end{align}
    and so to prove
    Theorem~\ref{theorem-PMF-strongly-zero-interval},
    it is sufficient to show that
    each of the sums in Equation~\ref{equation-PMF-nice-sums} 
    is zero.
    
    First consider $\sum_{\tau \in P} \mobp{\tau}$.
    $P$ is a closed interval, and so we have
    $\sum_{\tau \in P} \mobp{\tau} = 0$
    from the definition of the \mob function.

    Now consider
    $\sum_{\tau \in L_i} \mobp{\tau}$,
    with $i \in [1, k]$.
    Recall that
    $P 
    = \fullcover{\inflatesome{\gamma}{c}{\psi}} 
    = \permsetsome{\gamma}{c}{\psi}$,
    and
    $L_i^\prime 
    = \fullcover{\inflatesome{\gamma}{c}{\lambda_i}}
    = \permsetsome{\gamma}{c}{\lambda_i}$;
    and note that $L_i \subseteq L_i^\prime \setminus P$.    
    From this we can see that
    any permutation in $L_i^\prime \setminus P$ is
    in the set of permutations defined by
    \[
    M =
    \{ 
    \tau :
    \tau \in \permsetsome{\gamma}{c}{\upsilon},
    \upsilon \in ( \fullcover{\lambda_i} \setminus \fullcover{\psi} )
    \}.
    \]
    Since $\psi$ is the core of $\phi$, it follows that
    every permutation in $M$ is strongly zero.
    Now, since $L_i \subseteq L_i^\prime \setminus P$,
    we have that every permutation in $L_i$ 
    is strongly zero,
    and so
    $\sum_{\tau \in L_i} \mobp{\tau} = 0$.
        
    Finally, consider $\sum_{\tau \in R} \mobp{\tau}$.  
    Recall that $\pi = \inflatesome{\gamma}{c}{\phi}$. 
    We claim that every permutation in $R$ 
    contains an interval order-isomorphic to $\phi$.
    Any permutation $\alpha$ in $[1, \pi)$ is,
    by definition, contained in
    $\permsetsome{\gamma}{c}{\phi}$.
    It follows that $\alpha$ must be contained in 
    $
    \{
    \tau :  
    \tau \in \permsetfixed{\gamma}{c}{\upsilon},
    \upsilon \in \fullcover{\phi} 
    \}
    $.
    The permutations covered by $\phi$
    are $\psi, \lambda_1, \ldots, \lambda_k$,
    and so we have
    $\fullcover{\phi} = 
    \{ \phi  \} 
    \cup
    \fullcover{\psi}
    \cup
    \fullcover{\lambda_1}
    \cup
    \ldots
    \cup
    \fullcover{\lambda_k}
    $. 
    Since 
    $P 
    = \fullcover{\inflatesome{\gamma}{c}{\psi}} 
    =\permsetsome{\gamma}{c}{\psi}$,
    and
    $L_i
    = \fullcover{\inflatesome{\gamma}{c}{\lambda_i}} 
    =\permsetsome{\gamma}{c}{\lambda_i}$,
    it follows that 
    $
    R = \permsetfixed{\gamma}{c}{\phi}
    $,
    and so our claim is proved.
    Now, by the inductive hypothesis,
    we have that every permutation $\tau \in R$
    has $\mobp{\tau} = 0$,
    and so $\sum_{\tau \in R} \mobp{\tau} = 0$. 
\end{proof}
We note here that the set of nice permutations
is a subset of $\stronglyzeroset$, 
as a permutation $\pi$ with opposing adjacencies is 
in $\stronglyzeroset$, but $\pi$ may not be nice.
As an example, if $\pi=256143$,
then $\pi \in \stronglyzeroset$
since it has opposing adjacencies.
The permutations covered by $\pi$ are
$45132$, $25143$, $14532$ and $24513$.
Two of these, $25143$ and $24513$, are not strongly zero,
and so $\pi$ does not have a core, 
and therefore cannot be nice.

\subsection{Extending Theorem~\ref{theorem-PMF-strongly-zero-interval}}

It is natural to ask if we can extend  
Theorem~\ref{theorem-PMF-strongly-zero-interval}
to handle the case where the lower bound of the interval is not $1$.
In order to do so, we need some further definitions.

We define a \emph{$\sigma$-closure} of a permutation $\pi$,
written $\xfullcover{\sigma}{\pi}$,
to be the set of permutations contained in $\pi$
that also contain $\sigma$.  

Let $\stronglyzeroset_\sigma$ be the set of
permutations such that
if any permutation $\pi$ contains
an interval order-isomorphic 
to some $\tau \in \stronglyzeroset_\sigma$
then $\mobfn{\sigma}{\pi} = 0$.  

If $L$ is a set of permutations 
$\{\lambda_1, \ldots , \lambda_n\}$, then
the \emph{$\sigma$-ground} of $L$, $\xground{\sigma}{L}$,
is the set of permutations formed
by taking the $\sigma$-closure of each permutation in 
$\lambda_i$, and then removing any permutation
that is contained in $\stronglyzeroset_\sigma$.

We say that $\psi$ is a \emph{$\sigma$-core} of $\pi$
if every permutation in the $\sigma$-ground of $L$
is contained in $\psi$, i.e., 
$\xground{\sigma}{L} \subseteq \xfullcover{\sigma}{\psi}$.

Finally,
we say that $\pi$ is \emph{$\sigma$-nice}  
if
$\mobfn{\sigma}{\pi} = 0$ and 
$\pi$ has a $\sigma$-core.

We now have:
\begin{theorem}
    \label{theorem-MF-strongly-zero-interval}
    If $\phi$ is a $\sigma$-nice permutation,
    and $\pi$ is any permutation
    containing an interval order-isomorphic to $\phi$,
    then $\mobfn{\sigma}{\pi} = 0$.
\end{theorem}
\begin{proof}
    The proof is follows the same pattern as
    Theorem~\ref{theorem-PMF-strongly-zero-interval},
    replacing
    $\stronglyzeroset$ by $\stronglyzeroset_\sigma$,
    closure by $\sigma$-closure,
    ground by $\sigma$-ground,
    core by $\sigma$-core,
    nice by $\sigma$-nice,
    and where the lower bound of 
    an interval is $1$, 
    replacing the lower bound by $\sigma$.
    We omit the details for brevity.    
\end{proof}

\section{The proportion of permutations with $\mobp{\pi}=0$}
\label{section-bounds-for-zn}

Let $Z(n)$ be the proportion of permutations of length $n$
where the principal \mob function is zero.
Let $Z_{sz}(n)$ be the proportion of permutations
of length $n$ that are strongly zero.
Plainly, $Z(n) \geq Z_{sz}(n)$ for all $n$.
Our aim in this section is to find an
asymptotic lower bound for
$Z(n)$ by determining an asymptotic 
lower bound for $Z_{sz}(n)$.
To find this lower bound,
we count inflations of simple permutations
where the resulting permutation
has opposing adjacencies, 
and so is structurally zero.
We use a result from Albert and Atkinson~\cite{Albert2005}:
\begin{proposition}[{%
    Albert and Atkinson~\cite[Proposition 2]{Albert2005}}]
    \label{AA-inflations-of-simples}
    Let $\pi$ be any permutation. 
    Then there is a unique simple permutation 
    $\sigma$,
    and permutations 
    $\alpha_1, \ldots, \alpha_k$
    such that
    $\pi = \sigma[\alpha_1, \ldots, \alpha_k]$.
    If $\sigma \neq 12,21$, then
    $\alpha_1, \ldots, \alpha_k$
    are also uniquely determined by $\pi$.
    If $\sigma = 12$ or $21$, then
    $\alpha_1, \alpha_2$
    are unique so long as we require that 
    $\alpha_1$ is sum indecomposable or skew indecomposable respectively.
\end{proposition}

We will also need a 
result from Albert, Atkinson and Klazar~\cite{Albert2003}:
\begin{theorem}[{%
    Albert, Atkinson and Klazar~\cite[Theorem 5]{Albert2003}}]
    \label{AAK-number-of-simple-permutations}
    The number of simple permutations of length $n$,
    $S(n)$,
    is given by
    \[
    S(n) 
    =
    \dfrac{n!}{\e^2}
    \left(
    1 - \dfrac{1}{n} + \dfrac{2}{n(n-1)} +O(n^{-3})
    \right).
    \]
\end{theorem}

We will prove:
\begin{theorem}
    \label{theorem-percentage-PMF-zero}
    $Z(n)$ is, asymptotically,
    bounded below by \zpmfr.
\end{theorem}
\begin{proof}
    We find a lower bound for $Z_{sz}(n)$
    by counting 
    permutations that have
    opposing adjacencies.
    
    Let $n \geq 6$ be an integer;
    and let $k$ be an integer in the
    range $2, \ldots, \lfloor n/2 \rfloor$.
    Let $\sigma$ be a simple permutation with length $n-k$.
    We will count the number of ways 
    we can inflate $\sigma$ with $k$ adjacencies
    to obtain a permutation with length $n$
    that has opposing adjacencies.
    We can choose the positions to inflate in 
    $\binom{n-k}{k}$ ways.  There are
    $2^k$ distinct inflations by adjacencies, 
    and all but two have opposing adjacencies,
    thus the number of ways to inflate $\sigma$
    that result in a permutation 
    with opposing adjacencies is given by
    \[
    \binom{n-k}{k} (2^k - 2).    
    \]
    Since we are inflating
    simple permutations, it follows from
    Proposition~\ref{AA-inflations-of-simples}
    that the inflations are unique.
    
    For an inflation to contain an opposing
    adjacency, we need to inflate at least two points.
    Further, to obtain a permutation
    of length $n$ by inflating with adjacencies
    we can, at most, inflate $\lfloor n/2 \rfloor$ positions.
    Using Theorem~\ref{AAK-number-of-simple-permutations}, 
    we can say that  
    \[
    Z_{sz}(n) 
    \geq
    \dfrac{1}{n!}
    \sum_{k=2}^{\lfloor n/2 \rfloor}
    S(n-k)
    \binom{n-k}{k} (2^k - 2).
    \]
    
    Note now that as $n \to \infty$, $S(n) \to \frac{n!}{\e^2}$.  
    Let $P(n,k)$
    be the proportion of permutations of length $n$
    which are 
    inflations of simple permutations of length $n-k$,
    where $k$ positions are inflated by an adjacency,
    and the resulting permutation has 
    at least one opposing adjacency.    
    Then we have
    \begin{align*}
    \lim_{n \to \infty} P(n, 2) & =  
    \lim_{n \to \infty}
    \dfrac{1}{n!}
    \dfrac{(n-2)!}{\e^2} 
    \binom{n-2}{2} (2^2 - 2) \\
    & = 
    \lim_{n \to \infty}
    \dfrac{1}{\e^2} 
    \dfrac{(n-2)(n-3)}{n (n-1)} \\
    & = \dfrac{1}{\e^2}. 
    \end{align*}
    Similarly, we have
    \begin{align*}
    \lim_{n \to \infty} P(n, 3) & = \dfrac{1}{\e^2},  &
    \lim_{n \to \infty} P(n, 4) & = \dfrac{7}{12 \e^2},  \\
    \lim_{n \to \infty} P(n, 5) & = \dfrac{1}{4 \e^2},  &
    \lim_{n \to \infty} P(n, 6) & = \dfrac{31}{360 \e^2},  \\
    \lim_{n \to \infty} P(n, 7) & = \dfrac{1}{40 \e^2},  &
    \lim_{n \to \infty} P(n, 8) & = \dfrac{127}{20160 \e^2},  \\
    & \text{and} &
    \lim_{n \to \infty} P(n, 9) & = \dfrac{17}{12096 \e^2}.
    \end{align*}
    
    We now write
    \begin{align*}
    \lim_{n \to \infty}
    Z_{sz}(n) 
    & \geq
    \lim_{n \to \infty}    
    \dfrac{1}{n!}
    \sum_{k=2}^{\lfloor n/2 \rfloor}
    S(n-k)
    \binom{n-k}{k} (2^k - 2) \\
    & \geq
    \lim_{n \to \infty}    
    \dfrac{1}{n!}
    \sum_{k=2}^{9}
    S(n-k)
    \binom{n-k}{k} (2^k - 2) \\
    & = 
    \sum_{k=2}^{9}
    P(n, k) \\
    & = 0.3995299850
    \end{align*}
    and thus
    $Z(n)$ is, asymptotically,
    bounded below by \zpmfr.
\end{proof}
We used the first nine terms of the sum
$
\frac{1}{n!}
\sum_{k=2}^{\lfloor n/2 \rfloor}
S(n-k)
\binom{n-k}{k} (2^k - 2) 
$
as this gives us a lower bound for 
$Z_{sz}(n)$ to four significant figures.
We found that evaluating the first 100
terms
improves the lower bound to
$
0.3995764008
$,
and evaluating larger number of terms
makes no difference, to ten significant figures,
to the value obtained.
\begin{remark}
    Kaplansky~\cite{Kaplansky1945}
    provides an asymptotic expression for the
    probability that a permutation of
    length $n$ will have $k$ adjacencies.
    Corteel, Louchard and Pemantle~\cite{Corteel}
    show that the distribution is Poisson, 
    with parameter 2.  
    It is possible to find a lower bound for
    $Z(n)$ using a probabilistic argument based on
    these results.
    We found that, to four significant figures,
    the lower bound from this approach
    was still \zpmfr,
    so using this more general construction
    did not improve
    our lower bound. 
\end{remark}

\section{Concluding remarks}

\subsection{Permutations with non-opposing adjacencies}

Given Theorem~\ref{theorem-PMF-opposing-adjacencies},
it is natural to wonder if we can find a  
similar result that applies 
where a permutation has multiple adjacencies, 
but no opposing adjacencies.

We can find permutations that have 
multiple adjacencies, and do not
have
opposing adjacencies, where 
the principal \mob function value 
is non-zero.  
Table~\ref{table-count-of-non-opp-adj-permutations}
shows, for lengths $4 \ldots 12$, the number of 
permutations with multiple non-opposing adjacencies
broken down by whether
the value of the 
principal \mob function is zero or not.
\begin{table}[!h]
\[
\begin{array}{lrr}
\toprule
\text{Length} 
&  =0 & \neq 0 \\
\midrule
 4 &        6 &        4 \\
 5 &       26 &        8 \\
 6 &      170 &       38 \\
 7 &     1154 &      212 \\
 8 &     8954 &     1502 \\
 9 &    78006 &    13088 \\
10 &   757966 &   130066 \\
11 &  8132206 &  1436296 \\
12 & 95463532 & 17403612 \\
\bottomrule
\end{array}
\]
\caption{Number of permutations with non-opposing adjacencies,
    classified by the value of the 
    principal \mob function.}
\label{table-count-of-non-opp-adj-permutations}
\end{table}

This suggests that it might be possible to find
a result similar to
Theorem~\ref{theorem-PMF-opposing-adjacencies}
for some or all of these cases, although
any such result will clearly need some 
additional criteria that will exclude
permutations that have a non-zero
principal \mob function value.

\subsection{The number of strongly zero permutations}

In this paper we show that if a permutation is nice,
then it is strongly zero.  We also show that
permutations with opposing adjacencies are strongly zero.
There may be permutations that do not contain opposing adjacencies, 
and which are not nice, but are, nevertheless,
strongly zero.

We place the strongly zero permutations we can identify
into one of two categories:
\emph{obviously zero} permutations, 
which are those that contain
either an opposing adjacency, or
an interval that is order-isomorphic to a smaller
nice
permutation; 
and
\emph{new} permutations, 
which are those that are nice,
but not obviously zero.
As an example, $1243$ is obviously zero,
since it contains opposing adjacencies,
whereas
$12453$ is new.
The number of permutations for lengths 
3 to 10
in each of the above groups are shown in
Table~\ref{table-known-strongly-zero-permutations-3-to-10}.
\begin{table}[h!]
    \[
    \begin{array}{rrrrr}
    \toprule 
    \text{Length} &
    \text{Obviously zero} &
    \text{New} &
    \text{Obviously zero \%} &
    \text{New \%}  \\
    \midrule
    3  &      0 &     2 &  0.00 & 33.33 \\
    4  &     10 &     0 & 41.67 &  0.00 \\ 
    5  &     40 &    10 & 33.33 &  8.33 \\
    6  &    258 &    16 & 35.83 &  2.22 \\
    7  &   1570 &   144 & 31.15 &  2.86 \\
    8  &  11366 &   816 & 28.19 &  2.02 \\
    9  &  91254 &  6144 & 25.15 &  1.69 \\
    10 & 817506 & 50664 & 22.53 &  1.40 \\ 
    \bottomrule
    \end{array}
    \]
    \caption{Number of known strongly zero permutations with length $3, \ldots, 10$.}
    \label{table-known-strongly-zero-permutations-3-to-10}
\end{table}

The figures 
in Table~\ref{table-known-strongly-zero-permutations-3-to-10}
suggest that the number of 
permutations that are strongly zero grows as $n$ 
increases, which is what we would naturally expect.
There is also a suggestion,
on the basis of the limited numerical evidence,
that the proportion of permutations of length $n$
that are strongly zero is falling as $n$ increases.
We know, however, from the proof of 
Theorem~\ref{theorem-percentage-PMF-zero},
that this proportion is,
asymptotically, bounded below by
\zpmfr.  

We suggest a factor that might explain
this apparent contradiction. 
Our proof of 
Theorem~\ref{theorem-percentage-PMF-zero}
uses $\frac{n!}{\e^2}$ as the number of simple permutations
of length $n$.  
Table~\ref{table-compare-sn-vs-nfact-div-esquared}
compares the actual values of $S(n)$ against
the computed value of $\frac{n!}{\e^2}$
for $n = 4, \ldots, 10$, 
and, as can be seen, for 
these small values of $n$
there is a significant difference between the two values.
\begin{table}[!h]
    \[
    \begin{array}{lrrrrrrr}
    \toprule
    \text{n}         & 4 & 5& 6 & 7 & 8 & 9 & 10 \\
    \midrule
    S(n)             & 2 & 6 & 46 & 338 & 2926 & 28146 & 298526 \\
    \midrule
    \dfrac{n!}{\e^2} & 3.2 & 16.2 & 97.4 & 682 & 5456 & 49110 & 491104 \\
    \bottomrule
    \end{array}
    \]
    \caption{A comparison of $S(n)$ and $\dfrac{n!}{\e^2}$ for $n=4, \ldots, 10$.}
    \label{table-compare-sn-vs-nfact-div-esquared}
\end{table}
It is clear that the figures 
in Table~\ref{table-known-strongly-zero-permutations-3-to-10}
do not represent the asymptotic behaviour.

We say that a permutation $\pi$ is \emph{canonical}
if there is no symmetry of $\pi$
that is lexicographically smaller.
As an example, 
$125634$ is canonical,
as the other symmetries of this permutation are
$341256$,
$436521$, and
$652143$.
A file containing all of the known
canonical strongly zero permutations with 
length less than or equal to ten
is available 
from the second author.

\subsection{The asymptotic behaviour of $Z(n)$}

It is natural to wonder what the asymptotic behaviour of
$Z(n)$ is.
Based on numeric evidence supplied by 
Jason Smith~\cite{Smith2018}
for $1 \leq n \leq 9$, and calculations 
performed by the second author,
Table~\ref{table-zn-one-to-twelve} shows
the value of $Z(n)$ for $n = 1, \ldots, 12$.
\begin{table}[!h]
\[
\begin{array}{lr}
\toprule
\text{Length} & 
Z(n) \\
\midrule
 1 & 0.0000 \\ 
 2 & 0.0000 \\
 3 & 0.3333 \\
 4 & 0.4167 \\ 
 5 & 0.4833 \\
 6 & 0.5361 \\
 7 & 0.5742 \\
 8 & 0.5942 \\
 9 & 0.6019 \\
10 & 0.6040 \\
11 & 0.6034 \\
12 & 0.6021 \\
\bottomrule
\end{array}
\]
\caption{The value of $Z(n)$ for $n = 1, \ldots, 12$.}
\label{table-zn-one-to-twelve}
\end{table}

Based on this somewhat limited numeric evidence,
we conjecture that:
\begin{conjecture}
    \label{conjecture-PMF-zero-61}
    The proportion of permutations that have
    principal \mob function value
    equal to zero is
    bounded above by 0.6040.
\end{conjecture}
\begin{remark}
In our exploration of the \mob function,
we have noted that the behaviour of the function
can be erratic where the length of the permutation
is small, 
and Conjecture~\ref{conjecture-PMF-zero-61} 
may not, therefore, reflect
the asymptotic behaviour. 
\end{remark}

\paragraph*{Acknowledgements}

The computations in 
Section~\ref{section-bounds-for-zn} 
were performed using
Maple\texttrademark~\cite{Maple2016}.

\bibliographystyle{abbrv} 
\bibliography{../Bibliography}  
\end{document}